\documentclass[12pt]{article}
\usepackage{amsfonts}
\usepackage{amssymb}
\usepackage{mathrsfs}
\usepackage{srcltx}
\usepackage{amsmath}
\usepackage{amsthm}
\usepackage{amstext}
\usepackage{amsopn}
\usepackage{graphicx}

\newtheorem{theorem}{Theorem}[section]
\newtheorem{lemma}[theorem]{Lemma}
\newtheorem{proposition}[theorem]{Proposition}

\theoremstyle{definition}

\theoremstyle{remark}
\newtheorem{remark}[theorem]{Remark}

\numberwithin{equation}{section}

\newcommand{\ba}{\begin{array}}
\newcommand{\ea}{\end{array}}
\newcommand{\f}{\frac}

\newcommand{\Om}{\Omega}

\newcommand{\la}{\lambda}

\newcommand{\ds}{\displaystyle}

\begin{document}
\date{}
\title{ \bf\large{Asymptotic profiles of basic reproduction number for  epidemic spreading in heterogeneous environment}\thanks{S. Chen is supported by National Natural Science Foundation of China (No 11771109), and J. Shi is supported by US-NSF grants DMS-1715651 and DMS-1853598.}}
\author{Shanshan Chen\textsuperscript{1}\footnote{Email: chenss@hit.edu.cn},\ \ Junping Shi\textsuperscript{2}\footnote{Corresponding Author, Email: jxshix@wm.edu}
 \\
{\small \textsuperscript{1} Department of Mathematics, Harbin Institute of Technology,\hfill{\ }}\\
\ \ {\small Weihai, Shandong, 264209, P.R.China.\hfill{\ }}\\
{\small \textsuperscript{2} Department of Mathematics, College of William and Mary,\hfill{\ }}\\
\ \ {\small Williamsburg, Virginia, 23187-8795, USA.\hfill {\ }} }
\maketitle

\begin{abstract}
{The effect of diffusion rates on the basic reproduction number of  a general compartmental reaction-diffusion epidemic model in a heterogeneous environment is considered.  It is shown when the diffusion rates tend to zero, the limit of the basic reproduction number is the maximum value of the local reproduction number on the spatial domain. On the other hand when the diffusion rates tend to infinity,  the basic reproduction number tends to the spectral radius of the \lq\lq average\rq\rq~next generation matrix. These asymptotic limits of basic reproduction number hold for a class of general spatially heterogeneous  compartmental epidemic models, and they are  applied to a wide variety of examples.
}

\noindent {\bf{Keywords}}: Basic reproduction number; reaction-diffusion; \\ heterogeneous environment; compartmental epidemic models.
\end{abstract}

\section{Introduction}

In mathematical modeling of infectious diseases, the basic reproduction number $R_0$ is a key indicator for disease transmission. When $R_0<1$, the disease declines and eventually vanishes; and when $R_0>1$, the disease spreads in the population and an outbreak is possible \cite{1991book}. Roughly speaking, the basic reproduction number $R_0$ is the average number of healthy people infected by one contagious person over the course of the infectious period. In more mathematically rigorous terms, for ordinary differential equation epidemic models which is non-spatial,  $R_0$ is defined as the spectral radius of the next generation matrix \cite{Diekmann,Driessche}, which is established in  a general framework of compartmental disease transmission models. This definition is also  generalized to epidemic models with infinite-dimensional state space \cite{Thieme2009}.

As the environment in which the disease spreads is spatially heterogeneous, the transmission and spreading of the infectious disease is inevitably affected by the spatial structure and heterogeneity of the environment.  These factors can be incorporated into underlying mathematical models to show the effect of spatial heterogeneity on the disease transmission. The spatial structure and heterogeneity can be modeled in a discrete space using an ordinary differential equation patch model \cite{Allenpatch,Arino2003,lloyd1996spatial,Tien2015}, or they can be modeled in a continuous space using a reaction-diffusion-advection partial differential equation model \cite{Allen,CuiLou,WangZhao,WuZou}. The notion of the basic reproduction number is also extended to both classes of models. In particular a theory of basic reproduction numbers for general reaction-diffusion compartmental disease transmission models is recently developed in \cite{WangZhao}.


For spatially heterogeneous reaction-diffusion epidemic models, the basic reproduction number $R_0$ usually depends on the diffusion rates of populations. For example,
in the reaction-diffusion SIS epidemic model considered in \cite{Allen}:
\begin{equation*}
    \begin{cases}
    \ds\f{\partial I}{\partial t}=d_I\Delta I+\beta(x)SI-\gamma(x)I,&x\in\Omega,\;t>0,\\
\ds\f{\partial S}{\partial t}=d_S\Delta S-\beta(x)SI+\gamma(x)I,&x\in\Omega,\;t>0,\\
\ds\f{\partial I}{\partial \nu}=\ds\f{\partial S}{\partial \nu}=0, &x\in\partial \Omega,\;t>0,
    \end{cases}
\end{equation*}
where $\beta(x)$ is the transmission rate, $\gamma(x)$ is the removal rate, and $d_I,d_S$ are the diffusion rates of infectious and susceptible populations respectively,
it was shown that the basic reproduction number is defined as
\begin{equation}
R_0=\sup\left\{\ds\f{\int_\Omega \beta\phi^2dx }{\int_\Omega \left(d_I|\nabla \phi|^2+\gamma \phi^2\right)dx}:\phi\in H^1(\Omega),\phi\ne 0\right\}.
\end{equation}
Moreover it was shown in \cite{Allen} that $R_0$ has the following asymptotic profile with respect to the infectious population diffusion rate $d_I$:
\begin{equation}\label{prof}
\lim_{d_I\to 0}R_0= \max_{x\in\overline \Omega}\f{\beta (x)}{\gamma(x)},\;\;
\lim_{d_I\to \infty}R_0= \ds\f{\int_\Omega \beta dx }{\int_\Omega \gamma dx}.
\end{equation}
Notice that the quantity $\beta(x)/\gamma(x)$ is the local basic reproduction number at $x$ when there is no spatial movement, hence the global basic reproduction number tends to the maximum of local one as the diffusion rate tends to zero. On the other hand, the limit of basic reproduction number for large diffusion rate is the ratio of average transmission rate and average removal rate.
Similar asymptotic profiles for $R_0$ were also obtained in \cite{MagalWu} for several kinds of other spatially heterogeneous epidemic reaction-diffusion models. The results in \cite{MagalWu} are based on the fact that $R_0$ equals the spectral radius
of a product of the local basic reproduction number and strongly positive compact linear operators with spectral radii one.

In this paper, we aim to characterize limiting profiles of the basic reproduction number $R_0$ for general spatially heterogeneous reaction-diffusion compartmental epidemic models for small or large diffusion rates. 

We consider the following reaction-diffusion compartmental epidemic model
\begin{equation}\label{main}
\begin{cases}
\ds\f{\partial u_i}{\partial t}=d_i\Delta u_i+f_i(x,u), &x\in\Omega,\;t>0,\;1\le i\le n, \\
\ds\f{\partial u_i}{\partial \nu}=0, &x\in\partial \Omega,\;t>0,\;1\le i\le n,
\end{cases}
\end{equation}
which was proposed in \cite{WangZhao}.
Here $u_i$ is the density of the population in the $i$-th compartment, $d_i>0$ is constant and represents the diffusion coefficient of population $u_i$, $\Omega$ is a bounded domain in $\mathbb{R}^N$ ($N\ge1$) with smooth boundary $\partial \Omega$, $\nu$ is the outward unit normal vector at $x\in\partial \Omega$, and $f_i(x,u)$ is the reaction term in the $i$-th compartment. Moreover, $$f_i(x,u)=\mathcal F_i(x,u)-\mathcal V_i(x,u),$$
where $\mathcal F_i(x,u)$ is the input rate of newly infected individuals in the $i$-th compartment, $\mathcal V_i(x,u)=\mathcal V_i^-(x,u)-\mathcal V_i^+(x,u)$,
$\mathcal V_i^{+}(x,u)$ is the rate of transfer of individuals into the $i$-th compartment by all other means, and $\mathcal V_i^{-}(x,u)$ is the rate of transfer of individuals out of the $i$-th compartment. More biological explanation of model \eqref{main} could be found in \cite{WangZhao}.
In this paper, we will show the asymptotic profiles of $R_0$ for model \eqref{main} as $(d_1,\cdots,d_n)\to (0,\dots,0)$ and $(d_1,\dots,d_n)\to (\infty,\dots,\infty)$. Our results indict that the trend set in \cite{Allen,MagalWu} holds true for epidemic models in much more general setting: in small diffusion limit, the global basic reproduction number tends to the maximum of local basic reproduction number, and in large diffusion limit, the global basic reproduction number tends to some kind of spatial average of  local basic reproduction number.

There are extensive results on reaction-diffusion epidemic models. The asymptotic profiles of the endemic steady states were considered in
\cite{Allen,Peng2009,Peng2013,WuZou} and references therein, and the global dynamics of the epidemic models could be found in \cite{CaiWang,DengWu,LaiZou,WangFB2,WangFB1,WebWu2018,PengLiu,WangZhaoWang}.
The effect of diffusion and advection rates on $R_0$ and the stability of the disease-free steady state for a reaction-diffusion-advection epidemic model was considered in \cite{CuiLou}, see also \cite{CuiLamLou,HuangJin,KuoPeng,JinLewis} for reaction-diffusion-advection epidemic models. The definition of $R_0$ for time-periodic reaction-diffusion epidemic models was given in \cite{Ait,zhaoperi2,zhaoperi}, and the global dynamics for a time-periodic or almost  space periodic reaction-diffusion SIS epidemic model was studied in \cite{PengZhao,WangLiWang}.
The reaction-diffusion epidemic models with free boundary conditions were investigated in \cite{CaoLiYang,LinZhu2,LinZhu} and references therein, and reaction-diffusion epidemic models with time delays were also studied extensively, see e.g. \cite{BaiPengZhao,LouZhao,WuZhao}.

Throughout the paper, we use the following notations.
For $n\ge 1$,
\begin{equation}
\begin{split}
\mathbb R^n_+=&\{u=(u_1,\dots,u_n): u_i\ge0\;\;\text{for any}\;\;i=1,\dots,n\},\\
C\left(\overline\Omega,\mathbb R_+^n\right)=&\{(u_1(x),\dots,u_n(x)): u_i(x)\left(\in C(\overline\Omega,\mathbb R)\right)\ge0\;\;\text{for any}\;\;i=1,\dots,n\}.
\end{split}
\end{equation}
For a closed and linear operator $A$, we denote the spectral radius of $A$ by $r(A)$, the spectral set of $A$ by $\sigma(A)$, and the spectral bound of $A$ by
$$s(A):=\sup\{\mathcal R e\la:\la\in\sigma(A)\}.$$
Let $P=\left(P_{ij}\right)_{1\le i, j\le l}$ and $Q=\left(Q_{ij}\right)_{1\le i, j\le l}$ be $l\times l$ $(l\ge1)$ real-valued matrices, and let $Q(x)=\left(Q_{ij}(x)\right)_{1\le i, j\le l}$ be an $l\times l$ matrix-valued function.

$P\ge Q$ means $P_{ij}\ge Q_{ij}$ for each $1\le i, j\le l$.

$P> Q$ means $P_{ij}> Q_{ij}$ for each $1\le i,j\le l$.

$\ds\lim_{x\to x_0} Q(x)=Q$ means $\ds\lim_{x\to x_0} Q_{ij}(x)=Q_{ij}$ for each $1\le i, j\le l$.

The matrix $P$ is called positive if all entries of $P$ are non-negative and there exists at least one positive entry.

The matrix $P$ is called zero if all entries of $P$ are zero.

The matrix $P$ is called cooperative (or quasi-positive) if all off-diagonal entries of $P$ are non-negative, i.e., $P_{ij}\ge0$ for $i\ne j$.

Moreover, $(d_1,\dots,d_n)\to(0,\dots,0)$ means $\ds\max_{1\le j=1\le n}d_j\to0$.

$(d_1,\dots,d_n)\to(\infty,\dots,\infty)$ means $\ds\min_{1\le j\le n}d_j\to\infty$.

The remaining part of the paper is organized as follows. In Section 2, we show some preliminaries for further applications. In Section
2 and 3, We show the asymptotic profiles of $R_0$ for model \eqref{main} as $(d_1,\cdots,d_n)\to (0,\dots,0)$ and $(d_1,\dots,d_n)\to (\infty,\dots,\infty)$, respectively. In Section 4, we apply the  theoretical results to some concrete examples.

\section{Some preliminaries}
In this section, we recall the definition of basic reproduction number for reaction-diffusion epidemic models in \cite{WangZhao}.
Assume that the population $u=(u_1,\dots, u_n)^T$ of model \eqref{main} is divided into two types: infected compartments, labeled by $i=1,2,\dots,m$, and uninfected compartments, labeled by
$i=m+1,\dots,n$. We set
\begin{equation}\label{dsdi}
\begin{split}
&u_I=(u_1,\dots,u_m)^T,\;\;u_S=(u_{m+1},\dots,u_n)^T,\\
&d_I=(d_1,\dots,d_m)^T,\;\;d_S=(d_{m+1},\dots,d_n)^T,\\
&d_I\Delta u_I=(d_1\Delta u_1,\dots,d_m\Delta u_m)^T,\;\;d_S\Delta u_S=(d_{m+1}\Delta u_{m+1},\dots,d_{n}\Delta u_n)^T,\\
&f_I(x,u)=\left(f_1(x,u),\dots,f_m(x,u)\right)^T,\;\;f_S(x,u)=\left(f_{m+1}(x,u),\dots,f_n(x,u)\right)^T.
\end{split}
\end{equation}
Let $$U_s:=\{u\ge0:u_i=0 \;\text{for any}\;i=1,\dots,m\}$$ denote the set of all disease-free states of \eqref{main}, and assume that
model \eqref{main} has a disease-free steady state
\begin{equation}\label{equi}
u^0(x)=\left(0,\dots,0,u_{m+1}^0(x),\dots,u_{n}^0(x)\right)^T,
\end{equation}
where $u_i^0(x)>0$ for any $i=m+1,\dots,n$ and $x\in\overline\Omega$.
Define the following three matrices:
\begin{equation}\label{maFV}
\begin{split}
F(x,u)=&\left(F_{ij}(x,u)\right)_{1\le i,j\le m}=\left(\ds\f{\partial \mathcal F_i(x,u)}{\partial u_j}\right)_{1\le i,j\le m},\\ V(x,u)=&\left(V_{ij}(x,u)\right)_{1\le i,j\le m}=\left(\ds\f{\partial \mathcal V_i(x,u)}{\partial u_j}\right)_{1\le i,j\le m},\\
 M(x,u)=&\left(M_{ij}(x,u)\right)_{1\le i,j\le n-m}=\left(\ds\f{\partial f_{i+m}(x,u)}{\partial u_{j+m}}\right)_{1\le i,j\le n-m},
\end{split}
\end{equation}
and let
 \begin{equation}\label{Bini}
B:=d_I\Delta - V\left(x,u^0(x)\right).
\end{equation}
The following assumptions are imposed on model \eqref{main}: (see assumptions (A1)-(A6) in \cite{WangZhao})
\begin{enumerate}
\item [(A1)] For each $1\le i\le n$, functions $\mathcal F_i(x,u)$, $\mathcal V^+_i(x,u)$, $\mathcal V_i^-(x,u)$ are non-negative and continuously differentiable on
$\overline \Omega \times \mathbb R_+^n$.
\item [(A2)] If $u_i=0$, then $\mathcal V_i^-=0$.
\item [(A3)] $\mathcal F_i=0$ for $i>m$.
\item [(A4)] If $u\in U_s$, then $\mathcal F_i=\mathcal V_i^+=0$ for $i=1,\dots,m$.
\item [(A5)] $ M(x,u^0(x))$ is cooperative for any $x\in\overline\Omega$, and $$s\left(d_S\Delta + M(x,u^0(x))\right)<0.$$
\item [(A6)] $- V(x,u^0(x))$ is cooperative for any $x\in\overline\Omega$, and $s\left(B\right)=s(d_I\Delta - V\left(x,u^0(x)\right))<0$.
\end{enumerate}
Assumptions (A1)-(A6) are satisfied for most reaction-diffusion epidemic models.

Denote
\begin{equation}\label{X}
X=C\left(\overline\Omega,\mathbb{R}^m\right)\;\;\text{and}\;\;X_+=C\left(\overline\Omega,\mathbb{R}_+^m\right).
\end{equation}
$X$ is an ordered Banach space, and $X_+$ is a positive cone with nonempty interior.
Let $T(t)$ be the semigroup generated by $B$ on $X$, i.e., $T(t)$ is the solution semigroup associated with the following linear reaction-diffusion system:
\begin{equation}
\begin{cases}
\ds\f{\partial u_I}{\partial t}=d_I\Delta u_I-V\left(x,u^0(x)\right)u_I,&x\in\Omega,\;t>0,\\
\ds\f{\partial u_I}{\partial \nu}=0,&x\in\partial \Omega,\;t>0.
\end{cases}
\end{equation}
It follows from the comparison principle (see \cite[Theorem 3.12]{Thieme2009}) and assumption (A6) that $B$ is resolvent-positive, $T(t)$ is positive (i.e., $T(t)X_+\subset X_+$ for all $t>0$), $s(B)<0$, and
$-B^{-1}\phi=\int_0^\infty T(t)\phi dt$ for $\phi\in X$.
Note that $ F\left(x,u^0(x)\right)$ is a positive matrix, and it can also be viewed as a positive operator on $C\left(\overline\Omega,\mathbb{R}^m\right)$:
$$\phi\in C\left(\overline\Omega,\mathbb{R}^m\right)\mapsto F\left(x,u^0(x)\right)\phi.$$
Clearly,
the linear operator $B+F(x,u^0(x))$ is also resolvent-positive. Then it follows from \cite[Section 3]{WangZhao} (or \cite[Theorem 3.5]{Thieme2009}) that:
\begin{proposition}
Assume that (A1)-(A6) hold. Then the basic reproduction number is defined by
$$R_0=r\left(- F(x,u^0(x))B^{-1}\right).$$
Moreover, the following statements hold.
\begin{enumerate}
\item [(i)] $R_0-1$ has the same sign as $s\left(B+ F(x,u^0(x))\right)$.
\item [(ii)] If $R_0<1$, then $u^0(x)$ is locally asymptotically stable for system \eqref{main}.
\end{enumerate}
\end{proposition}

Next we recall  several results which will be used later.
First we have the following the comparison principle.
\begin{lemma}\label{compa}
Assume that $P_i(x)$ $(i=1,2)$ are $m\times m$ cooperative matrices for any $x\in\overline\Omega$, all entries of $P_i(x)$ $(i=1,2)$ are continuous, and $P_1(x)\ge P_2(x)$. Let $T_i(t)$ be the solution semigroup on $X$ (defined in Eq. \eqref{X}) associated with the following linear reaction-diffusion system:
\begin{equation}
\begin{cases}
\ds\f{\partial u_I}{\partial t}=d_I\Delta u_I+P_i(x)u_I,&x\in\Omega,\;t>0,\\
\ds\f{\partial u_I}{\partial \nu}=0,&x\in\partial \Omega,\;t>0,
\end{cases}
\end{equation}
where $d_I\Delta u_I$ is defined as in \eqref{dsdi}, and $d_i>0$ for $i=1,\dots,m$.
Then $T_1(t)\phi\ge T_2(t)\phi$ for any $\phi\in X_+$ and $t>0$.
\end{lemma}
\begin{proof}
Denote $U_i(x,t)=T_i(t)\phi$ for $\phi\in X_+$, and it follows from the comparison principle of cooperative parabolic systems that
$U_i(x,t)\ge0$ for any $(x,t)\in\overline\Omega\times (0,\infty)$ and $i=1,2$. Let $W(x,t)=U_1(x,t)-U_2(x,t)$, and then $W(x,t)$ satisfies
\begin{equation}
\begin{cases}
\ds\f{\partial W}{\partial t}=d_I\Delta W+P_2(x) W+\left(P_1(x)-P_2(x)\right)U_1,&x\in\Omega,\;t>0, \\
\ds\f{\partial W}{\partial \nu}=0, &x\in\partial\Omega,\;t>0,\\
W(x,0)=0,&x\in\Omega.
\end{cases}
\end{equation}
Note that $P_1(x)\ge P_2(x)$ and $U_1(x,t)\ge0$ for any $(x,t)\in\overline\Omega\times (0,\infty)$. Again it follows from the comparison principle of cooperative parabolic systems that
$W(x,t)\ge0$ for any $(x,t)\in\overline\Omega\times (0,\infty)$. This completes the proof.
\end{proof}

Secondly we recall the Krein-Rutmann theorem, (see \cite[Theorems 3.1 and 3.2]{Amann} or \cite[Theorem 2.5]{MagalWu}).
\begin{lemma}\label{kr}
\begin{enumerate}
\item [$(i)$] Suppose that $T:X\to X$ is a positive compact linear operator with positive spectral radius $r(T)$. Then $r(T)$ is an eigenvalue of $T$
with an eigenvector in $X_+\setminus\{0\}$.
\item [$(ii)$] Suppose that $T:X\to X$ is a strongly positive compact linear operator. Then $r(T)$ is positive and is a simple eigenvalue of $T$ with an eigenvector in  $Int (X_+)$, and there is no other eigenvalue with non-negative eigenvector. Moreover, if $S:X\to X$ is a linear operator such that $S-T$ is strongly positive, then $r(S)>r(T)$.
\end{enumerate}
\end{lemma}

Based on the Krein-Rutmann theorem in Lemma \ref{kr}, we have the following two results.
\begin{lemma}\label{krein}
Let $L_1$ and $L_2$ be bounded linear operators on $X$ (defined in Eq. \eqref{X}).
Assume that $L_1\phi\ge L_2\phi$ for any $\phi\in X_+$, and $L_2$ is a positive compact operator with positive spectral radius $r(L_2)$. Then $r(L_1)\ge r(L_2)$.
\end{lemma}
\begin{proof}
It follows from Lemma \ref{kr} that $r(L_2)$ is an eigenvalue of $L_2$, and there exists $\phi\in X_+\setminus\{0\}$ such that $\|\phi\|_\infty=1$ and $L_2\phi=r(L_2)\phi$.
Then $L_1^n\phi \ge r^n(L_2)\phi$, which implies that $\|L_1^n\|\ge r^n(L_2)$.
Therefore,
$r(L_1)=\ds\lim_{n\to\infty} \|L_1^n\|^{1/n}\ge r(L_2)$.
\end{proof}
Consider the following eigenvalue problem:
\begin{equation}\label{auei}
\begin{cases}
d_I\Delta \Phi-P(x)\Phi+aQ(x)\Phi=\la\Phi,& x\in\Omega,\\
\ds\f{\partial \Phi}{\partial \nu}=0,&x\in\partial \Omega,
\end{cases}
\end{equation}
where
\begin{equation}
\Phi=(\phi_1,\dots,\phi_m)^T,\;\;d_I\Delta\Phi=(d_1\Delta\phi_1,\dots,d_m\Delta\phi_m)^T,
\end{equation}
$a>0$, $d_i>0$ for $i=1,\dots,m$, and $P(x)=\left(P_{ij}(x)\right)_{1\le i,j\le m}$ and $Q(x)=\left(Q_{ij}(x)\right)_{1\le i,j\le m}$ are $m\times m$ matrices with continuous entries.
Recall that an eigenvalue $\la$ of \eqref{auei} is called the principal eigenvalue if $\la\in\mathbb R$ and for
any eigenvalue such that $\tilde \la \ne \la$, we have $\mathcal {R}e \tilde \la <\la$.
\begin{lemma}\label{monot}
Assume that
$-P(x)$ is cooperative, $Q(x)$ is positive for any $x\in\overline\Omega$, and for any $a\in(0,\infty)$, there exists $x_a\in\Omega$ such that $-P(x_a)+aQ(x_a)$ is irreducible. Let $\lambda(a)$ be the principal eigenvalue of \eqref{auei}. Then $\lambda(a)$ is strictly increasing for $a\in (0,\infty)$.
\end{lemma}
\begin{proof}
Since $-P(x)+aQ(x)$ is cooperative for any $x\in\overline\Omega$ and $a>0$, it follows from Lemma \ref{kr} that $\la(a)$ is well defined and
$$\la(a)=\sup\{\mathcal R e\la:\la\text{ is an eigenvalue of problem } \eqref{auei}\}.$$
Let $T^a(t)$ be the solution semigroup associated with the linear parabolic system
\begin{equation}
\begin{cases}
\ds\f{\partial V}{\partial t}=d_I\Delta V-P(x) V+aQ(x)V,&t>0,\; x\in\Omega, \\
\ds\f{\partial V}{\partial \nu}=0, &t>0,\; x\in\partial\Omega,\\
V(x,0)=V_0(x),&x\in\Omega.
\end{cases}
\end{equation}
Then it follows from \cite[Theorem 7.4.1]{Smith1995} that $T^a(t)$ is strongly positive and compact for any $a>0$ and $t>0$. Let $a_1>a_2$, $\Phi\in X_+\setminus\{0\}$, and
\begin{equation}
\begin{split}
U_1(x,t)=\left(U_1^{(1)}(x,t),\dots,U_1^{(m)}(x,t)\right)^T=T^{a_1}(t)\Phi,\\
U_2(x,t)=\left(U_2^{(1)}(x,t),\dots,U_2^{(m)}(x,t)\right)^T=T^{a_2}(t)\Phi.\\
\end{split}
\end{equation}
Then $U_1(x,t),U_2(x,t)>0$ for any $x\in\overline\Omega$ and $t>0$.
It follows from Lemma \ref{compa} that $U_1(x,t)\ge U_2(x,t)$ for any $x\in\overline\Omega$ and $t>0$.
Let $W(x,t)=U_1(x,t)-U_2(x,t)$, and we see that $U(x,t)$ satisfies
\begin{equation}
\begin{cases}
\ds\f{\partial W}{\partial t}=d_I\Delta W-P(x) W+a_2Q(x)W+(a_1-a_2)Q(x)U_1,&t>0,\;x\in\Omega, \\
\ds\f{\partial W}{\partial \nu}=0, &t>0,\;x\in\partial\Omega,\\
W(x,0)=0,&x\in\Omega.\\
\end{cases}
\end{equation}
Note that $Q(x)$ is positive for any $x\in\overline\Omega$, and $U_1(x,t)>0$ for any $x\in\overline\Omega$ and $t>0$. Then there exist $1\le i_1\le n$ and $x_0\in\Omega$ such that
$\ds\sum_{j=1}^n Q_{i_1j}(x_0)U^{(j)}_1(x_0,t)>0$ for any $t>0$, and consequently $W_{i_1}(x,t)>0$ for any $x\in\overline\Omega$ and $t>0$. Note that there exists
$x_{a_2}\in\Omega$ such that
$-P(x_{a_2}) +a_2Q(x_{a_2})$ is irreducible. Then there exists $i_2\ne i_1$ such that $-P_{i_2i_1}(x_{a_2})+a_2Q_{i_2i_1}(x_{a_2})>0$, which implies
that $W_{i_2}(x,t)>0$ for any $x\in\overline\Omega$ and $t>0$. Following the above process, we could obtain that
$W(x,t)>0$ for any $x\in\overline\Omega$ and $t>0$, which implies that $T_{a_1}(t)-T_{a_2}(t)$ is strongly positive for any $t>0$.
It follows from Lemma \ref{kr} that $$r\left(T_{a_1}(t)\right)=e^{\la(a_1)t}>r\left(T_{a_2}(t)\right)=e^{\la(a_2)t}\;\;\text{for any}\;\; t>0,$$
which implies that $\la(a_1)>\la(a_2)$. This completes the proof.
\end{proof}

\section{The effect of diffusion rates}

In this section, we show the asymptotic profile of $R_0$ for model \eqref{main} when all the diffusion rates are large or small.

\subsection{Small diffusion rates}
In this subsection, we consider the asymptotic profile of $R_0$ when $(d_1,\dots,d_n)\to(0,\dots,0)$. We first impose an additional assumption for this case:
\begin{enumerate}
\item [(A7)]  The disease-free steady state $(0,\dots,0,u_{m+1}^0(x),\dots,u_n^0(x))$ (defined in Eq. \eqref{equi}) satisfies
\begin{equation}\lim_{(d_{m+1},\dots,d_n)\to (0,\dots,0)}(u_{m+1}^0(x),\dots,u_n^0(x))=( c_{m+1}(x),\dots, c_n(x)) \;\;\text{in}\;\; C\left(\overline\Omega,\mathbb{R}^{n-m}\right),
\end{equation}
where $c_k(x)>0$ for any $x\in\overline\Omega$ and $k=m+1,\dots,n$.
\end{enumerate}
In the next section, we will show that this assumption is not restrictive, and it is satisfied for many kinds of epidemic models.
Denote
\begin{equation}\label{cx}
c(x)=(0,\dots,0,c_{m+1}(x),\dots, c_n(x))\in C\left(\overline\Omega,\mathbb{R}^n\right),
\end{equation}
and denote, for sufficiently small $\epsilon$ $(0<\epsilon<\min\{c_i(x):i=m+1,\dots,n,x\in\overline\Omega\})$,
\begin{equation}\label{lowupper2}
\begin{split}
&\mathcal D_\epsilon^c=\{(x,u_1,\dots,u_n):x\in\overline \Omega,\;u_i=0\text{ for }i=1,\dots,m, \\
&~~~~~~~~~~u_i\in[c_i(x)-\epsilon,c_i(x)+\epsilon] \text{ for }i=m+1,\dots,n\},\\
&\underline V^c_\epsilon=\left(\min_{(x,u)\in\mathcal D_\epsilon^c}V_{ij}(x, u)\right)_{1\le i,j\le m}=\left(\min_{(x,u)\in\mathcal D_\epsilon^c}\ds\f{\partial \mathcal V_i(x, u)}{\partial u_j}\right)_{1\le i,j\le m},\\
&\overline V^c_\epsilon=\left(\max_{(x,u)\in\mathcal D_\epsilon^c}V_{ij}(x, u)\right)_{1\le i,j\le m}=\left(\max_{(x,u)\in\mathcal D_\epsilon^c}\ds\f{\partial \mathcal V_i(x, u)}{\partial u_j}\right)_{1\le i,j\le m},\\
&\underline F^c_\epsilon=\left(\min_{(x,u)\in\mathcal D_\epsilon^c}F_{ij}(x, u)\right)_{1\le i,j\le m}=\left(\min_{(x,u)\in\mathcal D_\epsilon^c}\ds\f{\partial \mathcal F_i(x, u)}{\partial u_j}\right)_{1\le i,j\le m},\\
&\overline F^c_\epsilon=\left(\max_{(x,u)\in\mathcal D_\epsilon^c}F_{ij}(x, u)\right)_{1\le i,j\le m}=\left(\max_{(x,u)\in\mathcal D_\epsilon^c}\ds\f{\partial \mathcal F_i(x, u)}{\partial u_j}\right)_{1\le i,j\le m}.
\end{split}
\end{equation}
Since $\mathcal D^c_{\epsilon_1}\subset \mathcal D^c_{\epsilon_2}$ for $0\le\epsilon_1<\epsilon_2$, it follows that $\overline F^c_\epsilon$ and
$\overline V^c_\epsilon$ are monotone decreasing for $\epsilon\ge0$, and  $\underline F^c_\epsilon$ and
$\underline V^c_\epsilon$ are monotone increasing for $\epsilon\ge0$. We will show that these functions  $\overline F^c_\epsilon$,
$\overline V^c_\epsilon$, $\underline F^c_\epsilon$ and
$\underline V^c_\epsilon$ are also continuous for $\epsilon\ge0$ in the Appendix.

Clearly, for $\epsilon=0$, we have
\begin{equation*}
\begin{split}
&\underline V_0^c=\left(\min_{x\in\overline\Omega}V_{ij}(x,c(x))\right)_{1\le i,j\le m}=\left(\min_{x\in\overline\Omega}\ds\f{\partial \mathcal V_i(x,c(x))}{\partial u_j}\right)_{1\le i,j\le m},\\
&\overline V_0^c=\left(\max_{x\in\overline\Omega}V_{ij}(x,c(x))\right)_{1\le i,j\le m}=\left(\max_{x\in\overline\Omega}\ds\f{\partial \mathcal V_i(x,c(x))}{\partial u_j}\right)_{1\le i,j\le m},\\
&\underline F_0^c=\left(\min_{x\in\overline\Omega}F_{ij}(x,c(x))\right)_{1\le i,j\le m}=\left(\min_{x\in\overline\Omega}\ds\f{\partial \mathcal F_i(x,c(x))}{\partial u_j}\right)_{1\le i,j\le m},\\
&\overline F_0^c=\left(\max_{x\in\overline\Omega}F_{ij}(x,c(x))\right)_{1\le i,j\le m}=\left(\max_{x\in\overline\Omega}\ds\f{\partial \mathcal F_i(x,c(x))}{\partial u_j}\right)_{1\le i,j\le m}.
\end{split}
\end{equation*}
Now we show the asymptotic profile of $R_0$ as $(d_1,\dots,d_n)\to (0,\dots,0)$, and the method is motivated by the one in \cite{MagalWu}.
\begin{theorem}\label{bounds2}
Assume that (A1)-(A5) and (A7) hold, $$s\left(\overline V_0^c\right)<0, \; s\left(-\underline V_0^c\right)<0\;\;\text{and}\;\; r((\overline V_0^c)^{-1}\underline F_0^c)>0,$$
and there exists $\epsilon_0>0$ such that, for any $x\in\overline\Omega$, $-\overline V^c_{\epsilon_0}$ is cooperative and $\underline F^c_{\epsilon_0}$ is positive, where
$\underline V_\epsilon^c$, $\overline V_\epsilon^c$ and $\underline F_{\epsilon}^c$ are defined in \eqref{lowupper2}. If the matrix $-V(x,c(x))+aF(x,c(x))$ is irreducible for any $a>0$ and $x\in\overline \Omega$, where $c(x)$ is defined in \eqref{cx}, 
then
$$\ds\lim_{(d_1,\dots,d_n)\to(0,\dots,0)}R_0= R_0^c:=\max_{x\in\overline\Omega}\left[r\left(-V^{-1}(x,c(x))F(x,c(x))\right)\right].$$
\end{theorem}
\begin{proof}
{\bf Step 1.} We show that there exist positive constants $\underline R_0^c$, $\overline R_0^c$ and $C_2$ such that
$R_0\in[\underline R_0^c,\overline R_0^c]$ for any $d_1,\dots,d_m>0$ and $d_{m+1},\dots,d_n\in(0,C_2)$.\\
Since $-\overline V^c_{\epsilon_0}$ is cooperative and $\underline F^c_{\epsilon_0}$ is positive for any $x\in\overline\Omega$,
it follows from the monotonicity of $\overline F^c_\epsilon$,
$\overline V^c_\epsilon$, $\underline F^c_\epsilon$ and
$\underline V^c_\epsilon$ that
$-\overline V^c_\epsilon$ and  $-\underline V^c_\epsilon$ are cooperative, and $\overline F^c_{\epsilon_0}$ and $\underline F^c_{\epsilon}$ are positive for any $\epsilon\in[0,\epsilon_0]$.
Note that $\underline V^c_\epsilon$ and $\overline V^c_\epsilon$ are continuous with respect to $\epsilon$ (see Proposition \ref{pro1}), and
$$\lim_{\epsilon\to0}\underline V^c_\epsilon=\underline V_0^c\;\;\text{and}\;\;\lim_{\epsilon\to0}\overline V^c_\epsilon=\overline V_0^c.$$
It follows from \cite[Theorem 2.5.1]{Kato} that there exists $\epsilon_1\in(0,\epsilon_0)$ such that
$$s(-\underline V^c_\epsilon)<0,\; s(-\overline V^c_\epsilon)<0 \;\;\text{for any}\;\; \epsilon\in (0,\epsilon_1].$$
Similarly, $\left(\overline V_\epsilon^{c}\right)^{-1} \underline F^c_\epsilon$ is continuous with respect to $\epsilon$ for $\epsilon\in(0,\epsilon_1)$, and
there exists $\epsilon_2\in(0,\epsilon_1)$ such that $r\left(\left(\overline V_\epsilon^{c}\right)^{-1} \underline F^c_\epsilon\right)>0$ for any $\epsilon\in(0,\epsilon_2]$. It follows from (A7) that, for the above given $\epsilon_2>0$, there exists $C_2>0$ such that
\begin{equation*}
c_i(x)-\epsilon_2\le u^0_i(x)\le c_i(x)+\epsilon_2
\end{equation*}
for any $x\in\overline \Omega$, $d_{m+1},\dots,d_n\in(0,C_2)$ and $i=m+1,\dots,n$. Denote by $\overline T^c_{\epsilon_2} (t)$, $\underline T^c_{\epsilon_2} (t)$ and $T(t)$ the semigroups generated by $d_I\Delta-\overline V^c_{\epsilon_2}$, $d_I\Delta-\underline V^c_{\epsilon_2}$ and $d_I\Delta-V(x,u^0(x))$, respectively.
Note that
\begin{equation}
-\overline V^c_{\epsilon_2}\le -V(x,u^0(x))\le -\underline V^c_{\epsilon_2}
\end{equation}
for any $d_1,\dots,d_m>0$ and $d_{m+1},\dots,d_n\in(0,C_2)$, and $-\overline V^c_{\epsilon_2}$ is cooperative for any $x\in\overline\Omega$.
Then it follows from Lemma \ref{compa} that for any $\phi\in X_+$ (defined in Eq. \eqref{X}), $d_1,\dots,d_m>0$ and $d_{m+1},\dots,d_n\in(0,C_2)$,
\begin{equation}\label{estT}
\overline T^c_{\epsilon_2} (t)\phi\le T(t)\phi\le \underline T^c_{\epsilon_2} (t)\phi.
\end{equation}
Note that $$s\left(d_I\Delta-\underline V^c_{\epsilon_2}\right)=s(-\underline V^c_{\epsilon_2})<0,\;\;s\left(d_I\Delta-\overline V^c_{\epsilon_2}\right)=s\left(-\overline V^c_{\epsilon_2}\right)<0.$$
This, combined with Lemma \ref{kr} and the spectral mapping theorem, implies
that $r\left(T^c_{\epsilon_2} (t)\right),r\left(T^c_{\epsilon_2} (t)\right)\in(0,1)$.
Therefore, for any $d_1,\dots,d_m>0$ and $d_{m+1},\dots,d_n\in(0,C_2)$,
$ s\left(d_I\Delta-V(x,u^0(x))\right)<0$,
which implies that assumption (A6) is satisfied for any $d_1,\dots,d_m>0$ and $d_{m+1},\dots,d_n\in(0,C_2)$.
It follows from Eq. \eqref{estT} that
$$\underline F^c_{\epsilon_2}\int_0^\infty \overline T^c_{\epsilon_2} (t)\phi dt\le F(x, u^0(x))\int_0^\infty T(t)\phi dt\le \overline F^c_{\epsilon_2} \int_0^\infty\underline T^c_{\epsilon_2} (t)\phi dt.$$
It follows from \cite[Theorem 3.4]{WangZhao}
that $$r\left(\underline F^c_{\epsilon_2}\int_0^\infty \overline T^c_{\epsilon_2}dt (t)\right)=r\left(\left(\overline V^c_{\epsilon_2}\right)^{-1} \underline F^c_{\epsilon_2}\right)>0,$$
and $\underline F^c_{\epsilon_2}$ is positive and not zero for any $x\in\overline\Omega.$
Then we see from Lemma \ref{krein} that,
for any $d_1,\dots,d_m>0$ and $d_{m+1},\dots,d_n\in(0,C_2)$,
$$r\left(\left(\overline V^c_{\epsilon_2}\right)^{-1} \underline F^c_{\epsilon_2}\right)\le R_0\le r\left(\overline F^c_{\epsilon_2} \int_0^\infty\underline T^c_{\epsilon_2} (t)\phi dt\right)=r\left(\left(\underline V^c_{\epsilon_2}\right)^{-1} \overline F^c_{\epsilon_2}\right).$$
Let $\underline R^c_0=r\left(\left(\overline V^c_{\epsilon_2}\right)^{-1} \underline F^c_{\epsilon_2}\right)$ and $\overline R_0^c=r\left(\left(\underline V^c_{\epsilon_2}\right)^{-1} \overline F^c_{\epsilon_2}\right)$. This completes the proof for Step 1.

\noindent {\bf Step 2}. For any $x\in \overline{\Om}$, denote
\begin{equation}\label{lowupper2x}
\begin{split}
&\mathcal D^x=\{(u_1,\dots,u_n):\;u_i=0\text{ for }i=1,\dots,m, \\
&~~~~~~~~~~u_i\in[c_i(x)-\epsilon,c_i(x)+\epsilon] \text{ for }i=m+1,\dots,n\},\\
&\underline V^x_\epsilon=\left(\min_{u\in\mathcal D^x}V_{ij}(x, u)\right)_{1\le i,j\le m}=\left(\min_{u\in\mathcal D^x}\ds\f{\partial \mathcal V_i(x, u)}{\partial u_j}\right)_{1\le i,j\le m},\\
&\overline V^x_\epsilon=\left(\max_{u\in\mathcal D^x}V_{ij}(x, u)\right)_{1\le i,j\le m}=\left(\max_{u\in\mathcal D^x}\ds\f{\partial \mathcal V_i(x, u)}{\partial u_j}\right)_{1\le i,j\le m},\\
&\underline F^x_\epsilon=\left(\min_{u\in\mathcal D^x}F_{ij}(x, u)\right)_{1\le i,j\le m}=\left(\min_{u\in\mathcal D^x}\ds\f{\partial \mathcal F_i(x, u)}{\partial u_j}\right)_{1\le i,j\le m},\\
&\overline F^x_\epsilon=\left(\max_{u\in\mathcal D^x}F_{ij}(x, u)\right)_{1\le i,j\le m}=\left(\max_{u\in\mathcal D^x}\ds\f{\partial \mathcal F_i(x, u)}{\partial u_j}\right)_{1\le i,j\le m}.
\end{split}
\end{equation}
We show that, for sufficiently small $\epsilon>0$,
\begin{equation}\label{r00}
\begin{split}
&\tilde  R_0:= r\left(\left(d_I\Delta-\underline V^x_\epsilon\right)^{-1}\overline F^x_\epsilon\right)\to\tilde  R_0^0:=\max_{x\in\overline\Omega}r\left(\left(\underline V^x_\epsilon\right)^{-1}\overline F^x_\epsilon\right),\\
&\check R_0:=r\left(\left(d_I\Delta-\overline V^x_\epsilon\right)^{-1}\underline F^x_\epsilon\right)\to\check  R_0^0:=\max_{x\in\overline\Omega}r\left(\left(\overline V^x_\epsilon\right)^{-1}\underline F^x_\epsilon\right),
\end{split}
\end{equation}
as $d_I=(d_1,\dots,d_m)\to (0,\dots,0)$.

We can view matrices $-\underline V^x_\epsilon+a\overline F^x_\epsilon$ and $-\overline V^x_\epsilon+a\underline F^x_\epsilon$ as matrix-valued functions of
$(x,\epsilon,a)$. Then $-\underline V^x_\epsilon+a\overline F^x_\epsilon$ and $-\overline V^x_\epsilon+a\underline F^x_\epsilon$ are continuous and consequently uniformly continuous on $\overline\Omega\times[0,\epsilon_2]\times[1/\overline R_0,1/\underline R_0]$ (see Proposition \ref{pro1}).
This implies that
\begin{equation}
\begin{split}
\lim_{\epsilon\to0} (-\underline V^x_\epsilon+a\overline F^x_\epsilon)=-V(x,c(x))+aF(x,c(x))\\
\lim_{\epsilon\to0}(-\overline V^x_\epsilon+a\underline F^x_\epsilon)=-V(x,c(x))+aF(x,c(x))
\end{split}
\;\;\text{uniformly for}\;\; (x,a)\in\overline\Omega\times[1/\overline R_0,1/\underline R_0].
\end{equation}
Therefore,
there exists $\epsilon_3<\epsilon_2$ such that for any $\epsilon \in(0,\epsilon_3)$,
matrices $-\underline V^x_\epsilon+a\overline F^x_\epsilon$ and $-\overline V^x_\epsilon+a\underline F^x_\epsilon$ are irreducible for
any $a\in[1/\overline R_0,1/\underline R_0]$ and $x\in\overline\Omega$.
In this step, we always assume that $\epsilon\in(0,\epsilon_3]$. Clearly,
\begin{equation}
-\overline V^c_{\epsilon}\le -\underline V^x_\epsilon \le -\underline V^c_{\epsilon}.
\end{equation}
Noticing that $s(-\underline V^c_\epsilon),s(-\overline V^c_\epsilon)<0$ and $-\underline V^c_\epsilon$ is cooperative for any $x\in\overline\Omega$, we have
$s\left(d_I\Delta-\underline V^x_\epsilon\right)<0$.

Clearly, $\tilde R_0\in [\underline R_0,\overline R_0]$ and $\tilde R_0>0$. Let $\tilde \kappa=1/\tilde R_0$, and it follows from Lemma \ref{kr} that
$\tilde R_0$ is an eigenvalue of $\left(d_I\Delta-\underline V^x_\epsilon\right)^{-1}\overline F^x_\epsilon$ with a non-negative eigenvector
$\tilde \phi=(\tilde\phi_1,\dots,\tilde\phi_m)$. Clearly, $\tilde \kappa$ can be viewed as a function of $d_I$ (or respectively $(d_1,\dots,d_m)$), and
\begin{equation*}
d_I\Delta \tilde \phi-\underline V^x_\epsilon\tilde \phi+\tilde \kappa(d_1,\dots,d_m)\overline F^x_\epsilon\tilde \phi=0.
\end{equation*}
Let $\delta=\delta(d_1,\dots,d_m,a)$ be the principal eigenvalue of the auxiliary eigenvalue problem
\begin{equation}\label{aeigen}
d_I\Delta  \phi-\underline V^x_\epsilon\phi+a\overline F^x_\epsilon \phi=\delta\phi.
\end{equation}
Note that $\underline V^x_\epsilon+\tilde \kappa(d_1,\dots,d_m)\overline F^x_\epsilon$ is irreducible.
Then $\tilde \phi>0$, and $$\delta(d_1,\dots,d_m,\tilde \kappa(d_1,\dots,d_m))=0.$$It follows from \cite[Theorem 1.4]{LamLou} that
\begin{equation*}
\lim_{(d_1,\dots,d_m)\to(0,\dots,0)} \delta(d_1,\dots,d_m,a)=\max_{x\in\overline\Omega}\hat \delta \left(-\underline V^x_\epsilon+a\overline F^x_\epsilon \right).
\end{equation*}
Here $\hat\delta(Q)$ represents the eigenvalue of matrix $Q$ with greatest real part.
Define $$\delta(d_1,\dots,d_m,a):=\max_{x\in\overline\Omega}\hat \delta \left(-\underline V^x_\epsilon+a\overline F^x_\epsilon \right)$$ for $(d_1,\dots,d_m)=(0,\dots,0)$. Then, for each $a\in[1/\overline R_0,1/\underline R_0]$, $\delta(d_1,\dots,d_m,a)$ is a continuous function of $(d_1,\dots,d_m)$ on ${\rm Int}(\mathbb R^m_+)\cup\{(0,\dots,0)\}$.
It follows from Lemma \ref{monot} that $\delta(d_1,\dots,d_m,a)$ is strictly increasing in $a$ for each $(d_1,\dots,d_m)>(0,\dots,0)$.
Similarly, we see from Lemma \ref{monot} that, for each $x\in\overline \Omega$, $\hat \delta \left(-\underline V^x_\epsilon+a\overline F^x_\epsilon \right)$ is also strictly increasing in $a$. This implies that $\delta(d_1,\dots,d_m,a)$ is also strictly increasing in $a$ for $(d_1,\dots,d_m)=(0,\dots,0)$.
Since for any $x\in\overline\Omega$,
$$\underline V_{\epsilon_2}\le \underline V^x_\epsilon\le \overline V_{\epsilon_2} ,\;\; \underline F_{\epsilon_2}\le\overline F^x_\epsilon\le  \overline F_{\epsilon_2},$$
it follows from Step 1 that
$$\underline R_0\le r\left(\left(\underline V^x_\epsilon\right)^{-1}\overline F^x_\epsilon\right)\le\overline R_0,$$
for any $x\in\overline\Omega$, and
\begin{equation}\label{bonds}
\tilde R_0=r\left(\left(d_I\Delta-\underline V^x_\epsilon\right)^{-1}\overline F^x_\epsilon\right)\in[\underline R_0,\overline R_0].
\end{equation}
Noticing that, for each $x\in\overline \Omega$,
\begin{equation*}
\hat \delta \left(-\underline V^x_\epsilon+\ds\f{1}{r\left(\left(\underline V^x_\epsilon\right)^{-1}\overline F^x_\epsilon\right)}\overline F^x_\epsilon \right)=0.
\end{equation*}
Then the monotonicity of $\hat \delta \left(-\underline V^x_\epsilon+a\overline F^x_\epsilon \right)$ in $a$ implies that, for any $x\in\overline\Omega$,
\begin{equation*}
\hat \delta \left(-\underline V^x_\epsilon+\ds\f{1}{\tilde R_0^0}\overline F^x_\epsilon \right)\le0,
\end{equation*}
where $\tilde R_0^0$ is defined as in Eq. \eqref{r00},
and the equality holds if and only if $x$ achieves the maximum point of $r\left(\left(\underline V^x_\epsilon\right)^{-1}\overline F^x_\epsilon\right)$.
Therefore, the monotonicity of $\delta(0,\dots,0,a)$ implies that the unique zero of  $$\delta(0,\dots,0,a)=\max_{x\in\overline\Omega}\hat \delta \left(-\underline V^x_\epsilon+a\overline F^x_\epsilon \right)=0$$ on $[1/\overline R_0,1/\underline R_0]$ is $a=1/\tilde R_0^0$.

Now we claim that the first equation of \eqref{r00} holds. If it is not true, then $$\kappa(d_1,\dots,d_m)\not\to 1/ \tilde R_0^0\;\;\text{as}\;\; (d_1,\dots,d_n)\to (0,\dots,0).$$ Noticing that $\kappa(d_1,\dots,d_m)$ is bounded from Eq. \eqref{bonds}, we see that there exists a sequence $\left\{\left(d_1^{(j)},\dots,d_m^{(j)}\right)\right\}_{j=1}^\infty$ and $\kappa_0\left(\ne1/ \tilde R_0^0\right)\in[1/\overline R_0,1/\underline R_0]$ such that $$\left(d_1^{(j)},\dots,d_m^{(j)}\right)\to(0,\dots,0),\;\;\kappa_n:=\kappa\left(d_1^{(j)},\dots,d_m^{(j)}\right)\to \kappa_0\;\;\text{as}\;\; j\to \infty.$$
Without loss of generality, we assume that $\kappa_0<1/ \tilde R_0^0$. Then there exist $\tilde \epsilon$ and $j_0$ such that
$\kappa_0+\tilde \epsilon<1/ \tilde R_0^0$ and $\kappa_j<\kappa_0+\tilde \epsilon$ for any $j>j_0$. Then, for any $j>j_0$,
$$0=\delta\left(d_1^{(j)},\dots,d_m^{(j)},\kappa_j\right)<\delta\left(d_1^{(j)},\dots,d_m^{(j)},\kappa_0+\tilde \epsilon\right),$$
which yields
$$0\le\lim_{j\to\infty}\delta\left(d_1^{(j)},\dots,d_m^{(j)},\kappa_0+\tilde \epsilon\right)=\delta(0,\dots,0,\kappa_0+\tilde\epsilon)<0.$$
This is a contradiction, and therefore, the first equation of \eqref{r00} holds. Similarly, we can prove that the second equation of \eqref{r00} holds.

\noindent {\bf Step 3.} We show that $$\lim_{(d_1,\dots,d_n)\to(0,\dots,0)}R_0= \max_{x\in\overline\Omega}\left[r\left(-V^{-1}(x,c(x))F(x,c(x))\right)\right].$$

Clearly, $\left(\underline V_\epsilon^x\right)^{-1}\overline F_\epsilon^x $ can be viewed as a matrix-valued function of $(x,\epsilon)$, where $(x,\epsilon)\in\overline\Omega\times[0,\epsilon_3]$, and $\left(\underline V_\epsilon^x\right)^{-1}\overline F_\epsilon^x $ is continuous on
$\overline \Omega\times[0,\epsilon_3]$ (see Proposition \ref{pro1}). It follows from \cite[Section 2.5.7]{Kato} that $r\left(\left(\underline V_\epsilon^x\right)^{-1}\overline F_\epsilon^x\right)$ is continuous on $\overline \Omega\times[0,\epsilon_3]$, and consequently, $r\left(\left(\underline V_\epsilon^x\right)^{-1}\overline F_\epsilon^x\right)$ is uniformly continuous on $\overline \Omega\times[0,\epsilon_3]$. This implies that
$$\lim_{\epsilon\to0}r\left(\left(\underline V_\epsilon^x\right)^{-1}\overline F_\epsilon^x\right)=r\left(-V^{-1}(x,c(x))F(x,c(x))\right)\;\;\text{in}\;\;C(\overline\Omega).$$
Then $$\lim_{\epsilon\to0}\tilde R_0=\lim_{\epsilon\to0}\max_{x\in\overline\Omega}r\left[\left(\left(\underline V_\epsilon^x\right)^{-1}\overline F_\epsilon^x\right)\right]=R_0^c=\max_{x\in\overline\Omega}\left[r\left(-V^{-1}(x,c(x))F(x,c(x))\right)\right].$$
Similarly, we can prove that
$$\lim_{\epsilon\to0}\check R_0=R_0^c.$$
For any $\epsilon\in(0,\epsilon_3)$, there exists $\delta>0$ such that for any $d_{m+1},\dots,d_n<\delta$,
$$u^0_{i}(x)\in[c_i(x)-\epsilon,c_i+\epsilon] \;\;\text{for any}\;\;i=m+1,\dots,n\;\;\text{and}\;\;x\in\overline\Omega.$$
Then
$$\check R_0=r\left(\left(d_I\Delta-\overline V^x_\epsilon\right)^{-1}\underline F^x_\epsilon\right)\le R_0\le \tilde R_0= r\left(\left(d_I\Delta-\underline V^x_\epsilon\right)^{-1}\overline F^x_\epsilon\right)$$
for any $d_1,\dots,d_m>0$ and $d_{m+1},\dots,d_n<\delta$.
Therefore,
$$\check R_0\le \liminf_{(d_1,\dots,d_n)\to(0,\dots,0)}R_0\le \limsup_{(d_1,\dots,d_n)\to(0,\dots,0)} R_0\le\tilde R_0^0.$$
Taking $\epsilon\to 0$, we see that
$$\lim_{(d_1,\dots,d_n)\to(0,\dots,0)}R_0=R_0^c.$$
This completes the proof.
\end{proof}

\begin{remark}\label{re1}
In Theorem \ref{bounds2}, we assume that
there exists $\epsilon_0>0$ such that, for any $x\in\overline\Omega$, $-\overline V^c_{\epsilon_0}$ is cooperative and $\underline F^c_{\epsilon_0}$ is positive. In Section 4, we will show that in some concrete examples, any off-diagonal entry in $\overline V^c_{\epsilon}$, $\overline V^c_{\epsilon}$ either
equals to zero or is strictly positive, and any entry of
$\overline F^c_{\epsilon}$ or $\underline F^c_{\epsilon}$ is strictly positive. In that case we only need to assume that
$-\overline V^c_{0}$ is cooperative and $\underline F^c_{0}$ is positive to obtain results in Theorem \ref{bounds2}.
\end{remark}

\subsection{Large diffusion rates}
In this subsection, we consider the asymptotic profile of $R_0$ when $(d_1,\dots,d_n)\to(\infty,\dots,\infty)$. For this case, we impose an additional assumption:
\begin{enumerate}
\item [(A8)] The disease-free equilibrium $(0,\dots,0,u_{m+1}^0(x),\dots,u_n^0(x))$ (defined in Eq. \eqref{equi}) satisfies
\begin{equation}
\lim_{(d_{m+1},\dots,d_n)\to (\infty,\dots,\infty)} (u_{m+1}^0(x),\dots,u_n^0(x))= (\tilde u_{m+1},\dots,\tilde u_n) \;\;\text{in}\;\; C\left(\overline\Omega, \mathbb R^{n-m}\right),
\end{equation}
where $\tilde u_k$ is a positive constant for $k=m+1,\dots,n$.
\end{enumerate}
We will also show that this assumption is not restrictive and it is satisfied for many kinds of epidemic models in the next section.
Denote
\begin{equation}\label{tttu}
\tilde u=(0,\dots,0,\tilde u_{m+1},\dots,\tilde u_n)\in \mathbb{R}^n,\\
\end{equation}
and denote, for given sufficiently small $\epsilon$ $(0<\epsilon<\min\{\tilde u_i:i=m+1,\dots,n\})$,
\begin{equation}\label{lowupper}
\begin{split}
\mathcal D=&\{(u_1,\dots,u_n):u_i=0\text{ for }i=1,\dots,m,\\
&u_i\in[\tilde u_i-\epsilon,\tilde u_i+\epsilon] \text{ for }i=m+1,\dots,n\},\\
\underline V_\epsilon=&\left(\min_{x\in\overline\Omega,u\in\mathcal D}V_{ij}(x, u)\right)_{1\le i,j\le m}=\left(\min_{x\in\overline\Omega,u\in\mathcal D}\ds\f{\partial \mathcal V_i(x, u)}{\partial u_j}\right)_{1\le i,j\le m},\\
\overline V_\epsilon=&\left(\max_{x\in\overline\Omega,u\in\mathcal D}V_{ij}(x, u)\right)_{1\le i,j\le m}=\left(\max_{x\in\overline\Omega,u\in\mathcal D}\ds\f{\partial \mathcal V_i(x, u)}{\partial u_j}\right)_{1\le i,j\le m},\\
\underline F_\epsilon=&\left(\min_{x\in\overline\Omega,u\in\mathcal D}F_{ij}(x, u)\right)_{1\le i,j\le m}=\left(\min_{x\in\overline\Omega,u\in\mathcal D}\ds\f{\partial \mathcal F_i(x, u)}{\partial u_j}\right)_{1\le i,j\le m},\\
\overline F_\epsilon=&\left(\max_{x\in\overline\Omega,u\in\mathcal D}F_{ij}(x, u)\right)_{1\le i,j\le m}=\left(\max_{x\in\overline\Omega,u\in\mathcal D}\ds\f{\partial \mathcal F_i(x, u)}{\partial u_j}\right)_{1\le i,j\le m}.
\end{split}
\end{equation}
Similar to subsection 3.1, we could also prove that $\overline F_\epsilon$ and
$\overline V_\epsilon$ are monotone decreasing for $\epsilon\ge0$, and  $\underline F_\epsilon$ and
$\underline V_\epsilon$ is monotone increasing for $\epsilon\ge0$.
Moreover, when $\epsilon=0$,
\begin{equation*}
\begin{split}
&\underline V_0=\left(\min_{x\in\overline\Omega}V_{ij}(x,\tilde u)\right)_{1\le i,j\le m}=\left(\min_{x\in\overline\Omega}\ds\f{\partial \mathcal V_i(x,\tilde u)}{\partial u_j}\right)_{1\le i,j\le m},\\
&\overline V_0=\left(\max_{x\in\overline\Omega}V_{ij}(x,\tilde u)\right)_{1\le i,j\le m}=\left(\max_{x\in\overline\Omega}\ds\f{\partial \mathcal V_i(x,\tilde u)}{\partial u_j}\right)_{1\le i,j\le m},\\
&\underline F_0=\left(\min_{x\in\overline\Omega}F_{ij}(x,\tilde u)\right)_{1\le i,j\le m}=\left(\min_{x\in\overline\Omega}\ds\f{\partial \mathcal F_i(x,\tilde u)}{\partial u_j}\right)_{1\le i,j\le m},\\
&\overline F_0=\left(\max_{x\in\overline\Omega}F_{ij}(x,\tilde u)\right)_{1\le i,j\le m}=\left(\max_{x\in\overline\Omega}\ds\f{\partial \mathcal F_i(x,\tilde u)}{\partial u_j}\right)_{1\le i,j\le m}.\\
\end{split}
\end{equation*}
Now we show the asymptotic profile of $R_0$ as $(d_1,\dots,d_n)\to (\infty,\dots,\infty)$.

\begin{theorem}\label{bounds}
Assume that (A1)-(A5) and (A8) hold, $$s\left(\overline V_0\right)<0,\; s\left(-\underline V_0\right)<0\;\;\text{and}\;\; r(\overline V_0^{-1}\underline F_0)>0,$$
and there exists $\epsilon_0>0$ such that, for any $x\in\overline\Omega$, $-\overline V_{\epsilon_0}$ is cooperative and $\underline F_{\epsilon_0}$ is positive, where
$\underline V_\epsilon$, $\overline V_{\epsilon}$ and $\underline F_{\epsilon}$ are defined in Eq. \eqref{lowupper}.
Let
\begin{equation}\label{int}
\begin{split}
&\check V=\left(\int_\Omega V_{ij}(x,\tilde u)dx\right)_{1\le i,j\le m}=\left(\int_\Omega \ds\f{\partial \mathcal V_i(x,\tilde u)}{\partial u_j}dx\right)_{1\le i,j\le m},\\
&\check F=\left(\int_\Omega F_{ij}(x,\tilde u)dx\right)_{1\le i,j\le m}=\left(\int_\Omega \ds\f{\partial \mathcal F_i(x,\tilde u)}{\partial u_j}dx\right)_{1\le i,j\le m}.
\end{split}
\end{equation}
If $r(\check V^{-1}\check F)$ is the unique eigenvalue of $\check V^{-1}\check F$ with an eigenvector in $\mathbb R^m_+\setminus\{\mathbf 0\}$, then
$$\lim_{(d_1,\dots,d_n)\to(\infty,\dots,\infty)}R_0=r\left(\check V^{-1} \check F\right).$$
\end{theorem}

\begin{proof}
As in the Step 1 of Theorem \ref{bounds2}, we could prove that there exist positive constants $\underline R_0$, $\overline R_0$ and $C_2$ such that
$R_0\in[\underline R_0,\overline R_0]$ for any $d_1,\dots,d_m>0$ and $d_{m+1},\dots,d_n>C_2$.
Let $\kappa=1/R_0$, and $\kappa$ can be viewed as function of $(d_1,\dots,d_n)$. Since $k(d_1,\dots,d_n)$ is bounded for any $d_1,\dots,d_m>0$ and $d_{m+1},\dots,d_n>C_2$. Then, for any sequence $\{(d^{(j)}_1,\dots,d^{(j)}_n)\}_{j=1}^\infty$ satisfying
$(d^{(j)}_1,\dots,d^{(j)}_n) \to (\infty,\dots,\infty)$ as$j\to\infty,$
there exists a subsequence $\{(d^{(j_k)}_1,\dots,d^{(j_k)}_n)\}_{j=1}^\infty$ such that
$\ds\lim_{k\to\infty}\kappa\left(d^{(j_k)}_1,\dots,d^{(j_k)}_n\right)$ exists and is positive,
which is denoted by $\kappa^*$. For convenience, we denote $d^{(j_k)}_i$ by $d_i^{(k)}$ for each $k\ge1$ and $i=1,\dots,n$.
Without loss of generality, we assume that $d^{(k)}_i\ge C_2$ for any $k\ge 1$ and $i=m+1,\dots,n$.

Let $\hat\phi^{(k)}=(\hat\phi^{(k)}_1,\dots,\hat\phi^{(k)}_m)^T\ge(0,\dots,0)^T$ be the corresponding eigenvector of operator
$$-\left(d_I\Delta-V(x,u^0(x)\right)^{-1} F(x,u^0(x))$$ with respect to eigenvalue $ R_0(d^{(k)}_1,\dots,d^{(k)}_n)$, where $\|\hat\phi^{(k)}\|_\infty=1$ for each $k\ge1$. That is, for $i=1,\dots,m$,
$$\Delta \hat\phi_i^{(k)}+\ds\f{1}{d_i^{(k)}}\left[-\sum_{i=1}^{m}V_{ij}(x,u^0(x))\hat\phi_j^{(k)}+\kappa\left(d^{(k)}_1,\dots,d^{(k)}_n\right)\sum_{j=1}^mF_{ij}(x,u^0(x))\hat\phi_j^{(k)}\right]=0,$$
where $u^0(x)$ depends on $\left(d^{(k)}_{m+1},\dots,d^{(k)}_n\right)$.
Then it follows from the $L^p$ theory that there exists a subsequence $\{k_l\}_{l=1}^\infty$ such that
$\ds\lim_{l\to\infty}\hat\phi_i^{(k_l)}=c_i^\infty$ in $C(\overline\Omega,\mathbb R)$ for each $i=1,\dots,m$, where $c_i^{\infty}$ is a nonnegative constant, and $c^\infty:=(c_1^\infty,\dots,c_m^\infty)^T$ satisfies
$$|c^\infty|=1,\;\;\text{and}\;\;\check V c^\infty=\kappa^*\check Fc^\infty.$$
Then $1/\kappa^*=r(\check V^{-1}\check F)$. This completes the proof.
\end{proof}
\begin{remark}
We remark that there always exists a decomposition $$\left\{\left(\mathcal F_i(x,u),\mathcal V_i(x,u)\right)\right\}_{i=1}^n$$ of $\{f_i(x,u)\}_{i=1}^n$ such that $$f_i(x,u)=\mathcal F_i(x,u)-\mathcal V_i(x,u) \;\;\text {for}\;\;i=1,\dots,n,\;\;\text{and}\;\;rank(F)=1.$$
Consequently $r(\check V^{-1}\check F)$ is the unique eigenvalue of $\check V^{-1}\check F$ with an eigenvector in $\mathbb R^n_+\setminus\{\mathbf 0\}$. Moreover, different decompositions of $\{f_i(x,u)\}_{n=1}^\infty$ will not change the portion of parameter space that the disease vanishes or spreads.
Actually, if there exist two decompositions $$\left\{\left(\mathcal F^{(j)}_i(x,u),\mathcal V^{(j)}_i(x,u)\right)\right\}_{i=1}^n\;(j=1,2),$$
then there exist two basic reproduction numbers $R_0^{(1)}$ and $R_0^{(2)}$. It follows from \cite[Theorem 3.5]{Thieme2009} that
$R_0^{(1)}-1$ and $R_0^{(2)}-1$ have the same signs.
\end{remark}

\begin{remark}
In Theorem \ref{bounds}, we assume that
there exists $\epsilon_0>0$ such that, for any $x\in\overline\Omega$, $-\overline V_{\epsilon_0}$ is cooperative and $\underline F_{\epsilon_0}$ is positive. In Section 4, we will show that in some concrete examples, any off-diagonal entry in $\overline V_{\epsilon}$, $\overline V_{\epsilon}$
equals to zero or is strictly positive, and any entry of
$\overline F_{\epsilon}$ or $\underline F_{\epsilon}$ is strictly positive. Therefore we only need to show that
$-\overline V_{0}$ is cooperative and $\underline F_{0}$ is positive to obtain results in Theorem \ref{bounds}.
\end{remark}

\section{Applications}
In this section, we  give some examples to show that the general results in Theorems \ref{bounds2} and \ref{bounds} can be applied to many different reaction-diffusion epidemic models.

\subsection{Vector-host epidemic models}
We consider two vector-host epidemic models. The first is given by \cite{Web2017} to model the outbreak of Zika in Rio De Janerio:
\begin{equation}\label{zika}
\begin{cases}
\ds\f{\partial H_i}{\partial t}-\delta_1\Delta H_i=-\la(x) H_i+\sigma_1(x)H_u(x)V_i,&x\in\Omega,\;t>0,\\
\ds\f{\partial V_i}{\partial t}-\delta_2\Delta V_i=\sigma_2(x) V_uH_i-\mu(x)(V_u+V_i)V_i,&x\in\Omega,\;t>0,\\
\ds\f{\partial V_u}{\partial t}-\delta_3\Delta V_u=-\sigma_2(x) V_uH_i+\beta(x)(V_u+V_i)-\mu(x)(V_u+V_i)V_u,&x\in\Omega,\;t>0,\\
\ds\f{\partial H_i}{\partial \nu}=\ds\f{\partial V_i}{\partial \nu}=\ds\f{\partial V_u}{\partial \nu}, &x\in\partial \Omega,\;t>0,
\end{cases}
\end{equation}
where $H_u(x)$, $H_i(x,t)$, $V_i(x,t)$ and $V_u(x,t)$ are the densities of uninfected hosts, infected hosts, infected vectors and
uninfected vectors at space $x$ and time $t$, respectively, $\Omega$ is a bounded domain with smooth boundary $\partial \Omega$, $\nu$ is the outward unit normal vector on $\partial \Omega$,
$\delta_1,\delta_2,\delta_3$ are positive constants, and
$\la(x)$, $H_u(x)$, $\sigma_i(x)$ $(i=1,2)$, $\beta(x)$ and $\mu(x)$ are strictly positive and belong to $C^{\alpha}(\overline\Omega)$.
The asymptotic properties of $R_0$ for this model has been investigated in \cite{MagalWu}, see also \cite{WebWu2018} for the global dynamics.
We revisit it to show that the main results in Section 3 can  be applied to this model to determine the asymptotic behavior of basic reproduction number $R_0$.

Letting $$n=3, \;m=2 \;\;\text{and}\;\;(u_1,u_2,u_3)=(H_i,V_i,V_u),$$ we could use the framework in Section 3.
It follows from \cite{MagalWu} that model \eqref{zika} has a unique disease-free steady state $u^0(x)=\left(0,0,\hat V(x)\right)$, where $\hat V(x)$ satisfies
\begin{equation}
\lim_{\delta_3\to 0} \hat V(x)=\ds\f{\beta(x)}{\mu(x)} \;\;\text{and}\;\;\lim_{\delta_3\to\infty}\hat V(x)=\ds\f{\int_\Omega\beta(x) dx}{\int_\Omega\mu(x)dx} \;\;\text{in}\;\; C(\overline\Omega).
\end{equation}
This implies that assumptions (A7) and (A8) are satisfied.
For model \eqref{zika},
\begin{equation}\label{VF1}
V(x,u)=\left(\begin{array}{cc}
\la(x)&-\sigma_1(x)H_u(x)\\
0&\mu(x)u_3
\end{array}\right),\;\;F(x,u)=\left(\begin{array}{cc}
0&0\\
\sigma_2(x) u_3&0
\end{array}\right),
\end{equation}
where $u=(u_1,u_2,u_3)^T$,
and
\begin{equation}\label{B1}
B=\left(\delta_1\Delta,\delta_2\Delta\right)^T-V(x,u^0(x)).
\end{equation}
Then the basic reproduction number is given by
\begin{equation}\label{1r}
R_0=r\left(-F(x,u^0(x))B^{-1}\right).
\end{equation}
Moreover, for model \eqref{zika}, $$\tilde u=\left(0,0,\ds\f{\int_\Omega\beta dx}{\int_\Omega\mu dx}\right)\;\text{and} \;\;c(x)=\left(0,0,\ds\f{\beta(x)}{\mu(x)}\right),$$
and a direct computation implies that all the assumptions of Theorem \ref{bounds2} and \ref{bounds} are satisfied. Then
we have the following results.
\begin{proposition}
For model \eqref{zika}, the following statements hold.
\begin{enumerate}
\item [(i)] $$\lim_{(\delta_1,\delta_2,\delta_3)\to(\infty,\infty,\infty)}R_0=\ds\f{\ds\int_\Omega \sigma_1H_udx\int_\Omega\sigma_2dx}{\ds\int_\Omega\la dx\int_\Omega \mu dx}.$$
\item [(ii)] $$\lim_{(\delta_1,\delta_2,\delta_3)\to(0,0,0)}R_0=\ds\max_{x\in\overline \Omega} \ds\f{\sigma_1(x)\sigma_2(x)H_u(x)}{\la(x)\mu(x)}.$$
\end{enumerate}
\end{proposition}

Next we consider another vector-host epidemic model:
\begin{equation}\label{vector}
\begin{cases}
\ds\f{\partial I}{\partial t}=d_1\Delta I+\beta_s(x)SV-\left(b(x)+\gamma(x)\right)I,&x\in\Omega,\;t>0,\\
\ds\f{\partial V}{\partial t}=d_2\Delta V+\beta_m(x)MI-c(x)V, &x\in\Omega,\;t>0,\\
\ds\f{\partial S}{\partial t}=d_3\Delta S+\la_1(x)-b(x)S+\gamma(x)I-\beta_s(x)SV,&x\in\Omega,\;t>0,\\
\ds\f{\partial M}{\partial t}=d_4\Delta M+\la_2(x)-c(x)M-\beta_m(x)MI,&x\in\Omega,\;t>0,\\
\ds\f{\partial I}{\partial \nu}=\ds\f{\partial V}{\partial \nu}=\ds\f{\partial S}{\partial \nu}=\ds\f{\partial M}{\partial \nu}=0, &x\in\partial \Omega,\;t>0,
\end{cases}
\end{equation}
where  $I(x,t)$, $V(x,t)$, $S(x,t)$ and $M(x,t)$ are the densities of infected hosts, infected vectors, susceptible hosts and
susceptible vectors at space $x$ and time $t$, respectively,
$\Omega$ is a bounded domain with smooth boundary $\partial \Omega$, $\nu$ is the outward unit normal vector on $\partial \Omega$,
$d_1,d_2,d_3,d_4$ are positive constants, and
$\la_i(x)$ $(i=1,2)$, $\beta_s(x)$, $\beta_m(x)$, $b(x)$, $\gamma(x)$ and $c(x)$ are strictly positive and belong to $C^{\alpha}(\overline\Omega)$.
The model was originally an ODE model (i.e., $d_1=d_2=d_3=d_4=0$) proposed by Feng and Velasco-Hern{\'a}ndez \cite{Feng1997}, and $R_0$ of the ODE model was obtained in \cite{Feng1997,Driessche}.

Letting $$n=4, \;m=2\;\;\text{and}\;\;(u_1,u_2,u_3,u_4)=(I,V,S,M),$$ we could use the framework in Section 3.
The model \eqref{vector}
has a unique disease-free steady state $$u^0(x)=(0,0,\hat S(x),\hat M(x)),$$ where $\left(\hat S(x),\hat M(x)\right)$ satisfies
\begin{equation}
\begin{split}
&\lim_{(d_3,d_4)\to (0,0)} \left(\hat S(x),\hat M(x)\right)=\left(\ds\f{\la_1(x)}{b(x)},\ds\f{\la_2(x)}{c(x)}\right)\\
&\lim_{(d_3,d_4)\to (0,0)}\left(\hat S(x),\hat M(x)\right)=\left(\ds\f{\int_\Omega \la_1 dx}{\int_\Omega b dx},\ds\f{\int_\Omega \la_2d x}{\int_\Omega cdx}\right)
\end{split}\;\;\text{in}\;\; C(\overline\Omega,\mathbb{R}^2).
\end{equation}
This implies that assumptions (A7) and (A8) are satisfied.
A direct computation implies that, for model \eqref{vector},
\begin{equation}\label{VF2}
F(x,u)=\left(\begin{array}{cc}
0&\beta_s(x)u_3\\
\beta_m(x) u_4&0
\end{array}\right),\;\;V(x,u)=\left(\begin{array}{cc}
b(x)+\gamma(x)&0\\
0&c(x)
\end{array}\right)
\end{equation}
for $u=(u_1,u_2,u_3,u_4)^T$, and
\begin{equation*}
B=\left(d_1\Delta,d_2\Delta\right)^T-V(x,u^0(x)).
\end{equation*}
Then the basic reproduction number is also given by \eqref{1r}.
Finally for model \eqref{vector}, $$\tilde u=\left(0,0,\ds\f{\int_\Omega \la_1dx}{\int_\Omega b dx},\ds\f{\int_\Omega \la_2dx}{\int_\Omega cdx}\right)\;\text{and} \;\;c(x)=\left(0,0,\ds\f{\la_1(x)}{b(x)},\ds\f{\la_2(x)}{c(x)}\right).$$
It is easy to check that all the assumptions of Theorem \ref{bounds2} and \ref{bounds} are satisfied. Then
we have the following results.
\begin{proposition}
For model \eqref{vector}, the following statements hold.
\begin{enumerate}
\item [(i)]
$$
\lim_{(d_1,d_2,d_3,d_4)\to(\infty,\infty,\infty,\infty)}R_0=\sqrt{\ds\f{\int_\Omega\la_1dx\int_\Omega\la_2dx\int_\Omega\beta_sdx\int_\Omega\beta_mdx}{\int_\Omega b dx\left(\int_\Omega c dx\right)^2\int_\Omega(b+\gamma)dx}}.
$$
\item [(ii)]
$$
\lim_{(d_1,d_2,d_3,d_4)\to(0,0,0,0)}R_0=\max_{x\in\overline\Omega}\sqrt{\ds\f{\la_1(x)\la_2(x)\beta_s(x)\beta_m(x)}{b(x)c^2(x)(b(x)+\gamma(x))}}.
$$
\end{enumerate}
\end{proposition}

\subsection{Staged progression model}
In this subsection, we consider a staged progression model proposed in \cite{stanley}. This model has a single uninfected compartment, and the infected individuals could pass through several stages of the disease with changing infectivity. It could be applied to model the transmission of many disease, such as HIV/AIDS, see \cite{stanley}. The original model was an ODE model, and the reproduction number was obtained in \cite{GuoLi,Driessche}.
Here we consider the associated reaction-diffusion case:
\begin{equation*}
\begin{cases}
\ds\f{\partial I_1}{\partial t}=d_1\Delta I_1+h(N)\left(\sum_{k=1}^{m}\beta_k(x)SI_k\right)-(\nu_1(x)+\gamma_1(x))I_1,\;\;\;&x\in\Omega,\;t>0,\\
\ds\f{\partial I_i}{\partial t}=d_i\Delta I_i+\nu_{i-1}(x)I_{i-1}-(\nu_i(x)+\gamma_i(x))I_i,\;\;\;&x\in\Omega,\;t>0,\;2\le i\le m,\\
\ds\f{\partial I_{m+1}}{\partial t}=d_{m+2}\Delta I_{m+1}+\nu_{m}(x)I_{m}-\gamma_{m+1}I_{m+1},\;\;\;&x\in\Omega,\;t>0,\\
\ds\f{\partial S}{\partial t}=d_{m+1}\Delta S+\la(x)-b(x)S-h(N)\left(\sum_{k=1}^{m}\beta_k(x)SI_k\right),\;\;\;&x\in\Omega,\;t>0,\\
\ds\f{\partial S}{\partial \nu}=\ds\f{\partial I_i}{\partial \nu}=0, \;\;\;&x\in\partial \Omega,\;t>0,\;1\le i\le m,
\end{cases}
\end{equation*}
where $N=S+\ds\sum_{i=1}^m I_i$, $h(N)=N^{-\alpha}$  with $\alpha\in[0,1]$, $S(x,t)$ is the density of the susceptible individuals, $I_i(i=1,\dots,m+1)$ is the density of the infected individuals at stage
$i$, $\Omega$ is a bounded domain with smooth boundary $\partial \Omega$, $\nu$ is the outward unit normal vector on $\partial \Omega$,
$d_i$ $(i=1,\dots,m+2)$ are positive constants, and
$\la(x)$, $b(x)$, $\beta_i(x)$ $(i=1,\dots,m)$, $\nu_i(x)$ $(i=1,\dots,m)$, $\gamma_i(x)$ $(i=1,\dots,m+1)$ are strictly positive and belong to $C^{\alpha}(\overline\Omega)$.
Note that $I_{m+1}$ decouples from the others, and consequently we could  consider the following model
\begin{equation}\label{stage}
\begin{cases}
\ds\f{\partial I_1}{\partial t}=d_1\Delta I_1+h(N)\left(\sum_{k=1}^{m}\beta_k(x)SI_k\right)-(\nu_1(x)+\gamma_1(x))I_1,\;\;\;&x\in\Omega,\;t>0,\\
\ds\f{\partial I_i}{\partial t}=d_i\Delta I_i+\nu_{i-1}(x)I_{i-1}-(\nu_i(x)+\gamma_i(x))I_i,\;\;\;&x\in\Omega,\;t>0,\;2\le i\le m,\\
\ds\f{\partial S}{\partial t}=d_{m+1}\Delta S+\la(x)-b(x)S-h(N)\left(\sum_{k=1}^{m}\beta_k(x)SI_k\right),\;\;\;&x\in\Omega,\;t>0,\\
\ds\f{\partial S}{\partial \nu}=\ds\f{\partial I_i}{\partial \nu}=0, \;\;\;&x\in\partial \Omega,\;t>0,\;1\le i\le m.
\end{cases}
\end{equation}
Letting $$n=m+1, \;(u_1,\dots,u_m)=(I_1,\dots,I_m)\;\;\text{and}\;\; u_{m+1}=S,$$ we could use the framework in Section 3.
The model \eqref{stage}
has a unique disease-free steady state $$u^0(x)=(0,\dots,0,\hat S(x)),$$ where $\hat S(x)$ satisfies
\begin{equation}
\lim_{d_{m+1}\to 0} \hat S(x)=\ds\f{\la(x)}{b(x)}\;\;\text{and}\;\;\lim_{d_{m+1}\to \infty} \hat S(x)=\ds\f{\int_\Omega\la dx}{\int_\Omega bdx}\;\;\text{in}\;\; C(\overline\Omega,\mathbb{R}).
\end{equation}
This implies that assumptions (A7) and (A8) are satisfied.
For model \eqref{stage},
\begin{equation*}
V(x,u)=\left(V_{ij}(u)\right)_{1\le i,j\le m}\;\;\text{and}\;\;F(x,u)=\left(F_{ij}(x,u)\right)_{1\le i,j\le m},
\end{equation*}
where for $u=(u_1,\dots,u_{m})^T$,
\begin{equation*}
F_{ij}(x,u)=\begin{cases}
h\left(\sum_{k=1}^{m+1}u_k\right)\beta_j(x)u_{m+1}+h'\left(\sum_{k=1}^{m+1}u_k\right)\left(\sum_{k=1}^m\beta_k(x)u_k\right)u_{m+1}&i=1,1\le j\le m,\\
0& \text{otherwise},
\end{cases}
\end{equation*}
\begin{equation*}
V_{ij}(x,u)=\begin{cases}
\nu_i(x)+\gamma_i(x)&1\le i\le m,j=i,\\
-\nu_{i-1}(x)&2\le i\le m,j=i-1,\\
0& \text{otherwise},
\end{cases}
\end{equation*}
and
\begin{equation*}
B=\left(d_1\Delta,\dots,d_{m}\Delta\right)^T-V(x,u^0(x)).
\end{equation*}
And the basic reproduction number is given by \eqref{1r}. Also for model \eqref{stage}, $$\tilde u=\left(0,\dots,0,\ds\f{\int_\Omega \la dx}{\int_\Omega b dx}\right)\;\text{and} \;\;c(x)=\left(0,\dots,0,\ds\f{\la(x)}{b(x)}\right).$$
Since this model is more complex, we show that all the assumptions of Theorems \ref{bounds2} and \ref{bounds} are satisfied.
\begin{lemma}
The following statements hold.
\begin{enumerate}
\item [(i)] For any $a>0$ and $x\in\overline\Omega$, $-V(x,c(x))+aF(x,c(x))$ is irreducible.
\item [(ii)] $r(\check V^{-1}\check F)$ is the unique positive eigenvalue of $\check V^{-1}\check F$, where
$\check V^{-1}$ and $\check F$ are defined as in \eqref{int}.
\end{enumerate}
\end{lemma}
\begin{proof}
Let $Q_{ij}(x)=-V_{ij}(x,c(x))+aF_{ij}(x,c(x))$. Then a direct computation implies that
\begin{equation}
Q_{ij}(x)=\begin{cases}
a\beta_1(x)\ds\f{\la(x)}{b(x)}h\left(\f{\la(x)}{b(x)}\right)-\nu_1(x)-\gamma_1(x) &i=1,j=1,\\
a\beta_j(x)\ds\f{\la(x)}{b(x)}h\left(\f{\la(x)}{b(x)}\right)&i=1,\;2\le j\le m,\\
-\nu_i(x)-\gamma_i(x) &2\le i\le m,j=i,\\
\nu_{i-1}(x)&2\le i\le m, j=i-1,\\
0&otherwise.
\end{cases}
\end{equation}
For $i=1$, $2\le j\le m$, $$Q_{1j}=a\beta_j(x)\ds\f{\la(x)}{b(x)}h\left(\f{\la(x)}{b(x)}\right)\ne0,$$ and for any $2\le i\le m$, $j>i$, $$Q_{i(i-1)}\cdots Q_{21}Q_{1j}=a\nu_{i-1}(x)\cdots\nu_1(x)\beta_j(x)\ds\f{\la(x)}{b(x)}h\left(\f{\la(x)}{b(x)}\right)\ne0.$$
Similarly, for $1\le j\le m$, $i>j$, $$Q_{i(i-1)}Q_{(i-1)(i-2)}\cdots Q_{(j+1)j}=\nu_{i-1}(x)\cdots\nu_{j}(x)\ne0.$$
Therefore, $-V(x,c(x))+aF(x,c(x))$ is irreducible for any $a>0$ and $x\in\overline\Omega$. This completes the proof of part (i).

Let $\check V^{-1}=(\alpha_{ij})_{1\le i,j\le m}$. From \cite{Driessche} and a direct computation, we see that
\begin{equation}
\alpha_{ij}=\begin{cases}
0&1\le i\le m,\; j>i,\\
\ds\f{1}{\int_\Omega(\nu_i+\gamma_i)dx}&1\le i\le m,\; j=i,\\
\ds\f{\prod_{k=j}^{i-1}\int_\Omega\nu_kdx}{\prod_{k=j}^i\int_\Omega(\nu_k+\gamma_k)dx} &1\le i\le m, \; j<i.
\end{cases}
\end{equation}
Let $\check F\check V^{-1}=(\tilde \alpha_{ij})_{1\le i,j\le m}$. Then $\tilde a_{ij}=0$ for any $2\le i\le ,m$ and $1\le j\le m$,
and
\begin{equation}\label{tildea}
\tilde a_{11}=\left(\sum_{j=1}^m \f{\int_\Omega \beta_j dx\prod_{k=1}^{j-1}\int_\Omega \nu_kdx}{\prod_{k=1}^j\int_\Omega(\nu_k+\gamma_k)dx}\right)\ds\f{\int_\Omega \la dx}{\int_\Omega b dx}h\left(\ds\f{\int_\Omega \la dx}{\int_\Omega b dx}\right).
\end{equation}
Therefore, $r(\check V^{-1}\check F)=\tilde a_{11}$ is the unique positive eigenvalue of $\check V^{-1}\check F$.
This completes the proof of part (ii).
\end{proof}

The other assumptions of Theorems \ref{bounds2} and \ref{bounds} are easy to verified, and we omit the proof. Then we have the following results.
\begin{proposition}
Let $R_0$ be the basic reproduction number of model \eqref{stage}.
Then
\begin{enumerate}
\item [(i)]
$$\lim_{(d_1,\dots,d_{m+1})\to(0,\dots,0)}R_0=\max_{x\in\overline\Omega}\left(\sum_{j=1}^m \f{ \beta_j(x) \prod_{k=1}^{j-1} \nu_k(x)}{\prod_{k=1}^j(\nu_k(x)+\gamma_k(x))}\right)\f{\la(x)}{b(x)}h\left(\f{\la(x)}{b(x)}\right).$$
\item [(ii)] $$\lim_{(d_1,\dots,d_{m+1})\to(\infty,\dots,\infty)}R_0=\left(\sum_{j=1}^m \f{\int_\Omega \beta_j dx\prod_{k=1}^{j-1}\int_\Omega \nu_kdx}{\prod_{k=1}^j\int_\Omega(\nu_k+\gamma_k)dx}\right)\ds\f{\int_\Omega \la dx}{\int_\Omega b dx}h\left(\ds\f{\int_\Omega \la dx}{\int_\Omega b dx}\right).$$
\end{enumerate}
\end{proposition}

\section{Appendix}
In this part, we prove a result that verifies the continuity of functions $\underline F_\epsilon^c$, $\overline F_\epsilon^c$, $\underline V_\epsilon^c$, $\overline V_\epsilon^c$, $\underline F_\epsilon^x$, $\overline F_\epsilon^x$, $\underline V_\epsilon^x$ and $\overline V_\epsilon^x$, which are defined in
Eqs. \eqref{lowupper2} and \eqref{lowupper2x}.
\begin{proposition}\label{pro1}
Let $f(x,u)\in C(\overline\Omega\times \mathbb R,\mathbb R)$ and $c(x)\in C(\overline\Omega,\mathbb R)$, where $\Omega$ is a bounded domain in $\mathbb R^N$ $(N\ge1)$. Denote
$$\mathcal D_\epsilon^c=\{(x,u):x\in\overline\Omega, u\in[c(x)-\epsilon,c(x)+\epsilon]\},\;\;\mathcal D_\epsilon^x=\{u: u\in[c(x)-\epsilon,c(x)+\epsilon]\},$$
and
$$
H(\epsilon)=\max_{(x,u)\in \mathcal D_\epsilon^c}f(x,u),\;\;G(x,\epsilon)=\max_{u\in \mathcal D_\epsilon^x}f(x,u).
$$
Then $H(\epsilon)\in C([0,1],\mathbb R)$ and $G(x,\epsilon)\in C(\overline\Omega\times[0,1],\mathbb R)$.
\end{proposition}
\begin{proof}
We first consider the continuity of $ H(\epsilon)$. Let
$$C_1:=\min_{x\in\overline\Omega}c(x)-2 \;\;\text{and}\;\; C_2=:\max_{x\in\overline\Omega}c(x)+2.$$
The continuity of $f(x,u)$ implies that $f(x,u)$ is uniformly continuous on $\overline\Omega\times [C_1,C_2]$.
Then, for any give $\gamma>0$, there exists $\delta>0$ such that, for any $(x_1,u_1), (x_2,u_2)\in\overline\Omega\times[C_1,C_2]$ satisfying $|x_1-x_2|<\delta$ and $|u_1-u_2|<\delta$,
\begin{equation}\label{rf}
|f(x_1,u_1)-f(x_2,u_2)|<\gamma.
\end{equation}
Assume that $0\le\epsilon_1<\epsilon_2\le1$ and $\epsilon_2-\epsilon_1<\delta$. Clearly, $H(\epsilon_1)\le H(\epsilon_2)$.
Noticing that $\mathcal D_{\epsilon_2}^c$ is compact, we see that there exists $(x_0,u_0)\in \mathcal D_{\epsilon_2}^c$ such that
$H(\epsilon_2)=f(x_0,u_0)$. Then there exists $(x_0,u_1)$ such that $(x_0,u_1)\in\mathcal D_{\epsilon_1}^c$ and $|u_1-u_0|<\delta$.
It follows from Eq. \eqref{rf} that
$f(x_0,u_0)< f(x_0,u_1)+\gamma$, which implies that
$H(\epsilon_2)<H(\epsilon_1)+\gamma$.
Then exchanging the position of $\epsilon_1$ and $\epsilon_2$, we can also obtain that,
for any $0\le\epsilon_2<\epsilon_1\le1$ and $\epsilon_1-\epsilon_2<\delta$,
$$H(\epsilon_2)\le H(\epsilon_1)\le H(\epsilon_2)+\gamma.$$
Therefore, for any given $\gamma>0$, there exists $\delta>0$ such that,
for any $\epsilon_1,\epsilon_2\in[0,1]$ satisfying $|\epsilon_1-\epsilon_2|<\delta$,
$$|H(\epsilon_1)-H(\epsilon_2)|<\gamma.$$ This implies that $H(\epsilon)\in C([0,1],\mathbb R)$.

Then we consider the continuity of $ G(x,\epsilon)$. Note that $c(x)$ is continuous. Then,
for the above $\delta$, there exists $\delta_1\in(0,\delta)$ such that, for any $x_1,x_2\in\overline\Omega$ satisfying $|x_1-x_2|<\delta_1$,
$$|c(x_1)-c(x_2)|<\delta/2.$$
Clearly, if $|\epsilon_1-\epsilon_2|<\delta/2$ and $|x_1-x_2|<\delta_1$, then
\begin{equation}\label{ccc}
|c(x_2)+\epsilon_2-c(x_1)-\epsilon_1|<\delta\;\;\text{and}\;\;|c(x_2)-\epsilon_2-c(x_1)+\epsilon_1|<\delta.
\end{equation}
Choose $(x_1,\epsilon_1),(x_2,\epsilon_2)\in\overline\Omega\times[0,1]$ satisfying
$$|x_1-x_2|, \;|\epsilon_1-\epsilon_2|<\delta_2,$$
where $\delta_2:=\min\{\delta/2,\delta_1\}$. Clearly, there exists $u_1\in[c(x_1)-\epsilon_1,c(x_1)+\epsilon_1]$
such that $G(x_1,\epsilon_1)=f(x_1,u_1)$. Then we claim that
$$G(x_1,\epsilon_1)< G(x_2,\epsilon_2)+\gamma,$$
and the proof is divided into two cases.\\
{\bf Case 1.} $u_1\in[c(x_2)-\epsilon_2,c(x_2)+\epsilon_2]$. \\
Since $|x_1-x_2|<\delta_2<\delta$, it follows from Eq. \eqref{rf} that 
$$G(x_1,\epsilon_1)=f(x_1,u_1)<f(x_2,u_1)+\gamma\le G(x_2,\epsilon_2)+\gamma.$$
{\bf Case 2.} $u_1\not\in[c(x_2)-\epsilon_2,c(x_2)+\epsilon_2]$. \\
Then $u_1> c(x_2)+\epsilon_2$ or $u_1<c(x_2)-\epsilon_2$. We only consider the case of $u_1> c(x_2)+\epsilon_2$, and the other case could be proved similarly. Then $ c(x_2)+\epsilon_2<u_1\le c(x_1)+\epsilon_1$, This, combined with Eq. \eqref{ccc}, implies that
$|c(x_2)+\epsilon_2-u_1|<\delta$. Then it follows from Eq. \eqref{rf} that
$$G(x_1,\epsilon_1)=f(x_1,u_1)<f\left(x_2,c(x_2)+\epsilon_2\right)+\gamma\le G(x_2,\epsilon_2)+\gamma.$$
Then exchanging the positions of $(x_1,\epsilon_1)$ and $(x_2,\epsilon_2)$, we also have
$$G(x_2,\epsilon_2)< G(x_1,\epsilon_1)+\gamma.$$ This implies that for any given $\gamma>0$, there exists $\delta_2>0$ such that,
for any $$(x_1,\epsilon_1), (x_2,\epsilon_2)\in\overline\Omega\times[0,1]$$ satisfying $|x_1-x_2|<\delta_2$ and $|\epsilon_1-\epsilon_2|<\delta_2$,
\begin{equation}\label{rf2}
|G(x_1,\epsilon_1)-G(x_2,\epsilon_2)|\le\gamma.
\end{equation}
This completes the proof.
\end{proof}

\end{document}